\documentclass[review,3p,times]{elsarticle}
\usepackage{array}
\usepackage{verbatim}
\usepackage{bm}
\usepackage{multirow}
\usepackage{amsmath}
\usepackage{amssymb}
\usepackage{graphicx}
\usepackage{esint}
\usepackage{comment}
\usepackage{float}
\usepackage[subfigure]{graphfig}
\usepackage[unicode=true,
 bookmarks=false,
 breaklinks=false,pdfborder={0 0 1},colorlinks=false]
 {hyperref}
\usepackage{tikz}
\usetikzlibrary{shapes.geometric, arrows, positioning, matrix}

\makeatletter

\usepackage{tabularx}\usepackage{subfigure}\usepackage{subfigure}\usepackage{color}\usepackage{textcomp}\usepackage{latexsym}

\numberwithin{equation}{section}

\usepackage{setspace}

\usepackage{appendix}

\usepackage{bm}
\usepackage{amsthm}
\newtheorem{definition}{\bf Definition}

\newtheorem{proposition}{\bf Proposition}

\newtheorem{assumption}{\bf Assumption}

\journal{}

\makeatother

\begin{document}

\title{On the convection boundedness of numerical schemes across discontinuities}

\author[ad1]{Xi Deng}

\author[ad2]{Zhen-hua Jiang \corref{cor}}

\author[ad1]{Omar K. Matar}

\author[ad2]{Chao Yan}

\address[ad1]{Department of Chemical Engineering,Imperial College London,SW7 2AZ, United Kingdom.}

\address[ad2]{College of Aeronautics Science and Engineering, Beijing University of Aeronautics and Astronautics, Beijing 100191, PR China.}

\cortext[cor]{Corresponding author: Dr. Zhen-Hua Jiang (Email: jiangzhenhua@buaa.edu.cn)}

\begin{abstract}

This short note introduces a novel diagnostic tool for evaluating the convection boundedness properties of numerical schemes across discontinuities. The proposed method is based on the convection boundedness criterion and the normalised variable diagram. By utilising this tool, we can determine the CFL conditions for numerical schemes to satisfy the convection boundedness criterion, identify the locations of over- and under-shoots, optimize the free parameters in the schemes, and develop strategies to prevent numerical oscillations across the discontinuity. We apply the diagnostic tool to assess representative discontinuity-capturing schemes, including THINC, fifth-order WENO, and fifth-order TENO, and validate the conclusions drawn through numerical tests. We further demonstrate the application of the proposed method by formulating a new THINC scheme with less stringent CFL conditions.
\end{abstract}


\maketitle

\section{Introduction}
Numerical simulations of complex flow systems present significant challenges due to the presence of discontinuities, such as material interfaces in multiphase flows, reaction fronts in combustion, and shock waves in supersonic flows. Numerical methods for resolving discontinuities are generally categorized into tracking and capturing schemes. Discontinuity-capturing schemes are widely used due to their flexibility and ability to extend to higher-order accuracy. However, designing high-resolution discontinuity-capturing schemes is challenging, as Godunov's theorem states that no linear scheme of higher than second order can maintain monotonicity.

Over the decades, significant efforts have been made to overcome Godunov's barrier and develop non-linear, high-resolution discontinuity-capturing schemes. High-order shock-capturing schemes, such as WENO (Weighted Essentially Non-Oscillatory) \cite{liuweno,jiang96,wenoZ_Borges2008}, CWENO (Central WENO) \cite{levy2000compact,capdeville2008central}, and TENO (Targeted Essentially Non-Oscillatory) \cite{fu2016family,fu2017targeted}, have been successfully developed. Recent advancements \cite{wenoZ_Don2013,tsoutsanis2021cweno,dumbser2017central,zhu2018new} have further improved the resolution and robustness of these schemes. Interface-capturing schemes, such as THINC (Tangent Hyperbola for Interface Capturing) \cite{xiao2005simple}, have been applied to compressible flows to enhance the resolution of discontinuous flow structures, such as contact discontinuities, through the BVD (Boundary Variation Diminishing) algorithm \cite{deng2018high,deng2018limiter,deng2019fifth} or discontinuity-detecting criterion \cite{liang2022fifth}. Recent work \cite{deng2023unified,deng2023new} has demonstrated that existing three-cell-based non-linear schemes can be unified into a single framework, from which a new high-resolution scheme, named ROUND (Reconstruction Operator on Unified Normalized-variable Diagram), has been proposed.    

Significant efforts have also been made to understand and optimize the numerical properties of non-linear discontinuity-capturing schemes, such as their spectral properties \cite{adr} and stability \cite{wang2007linear}. A general framework based on a quantitative error metric for evaluating shock-capturing schemes has also been developed \cite{zhao2019general}. To quantify the overshoot error as a function of the CFL number, recent work \cite{zhang2021numerical} introduced error metrics for non-linear shock-capturing schemes. However, limited research has addressed the convection boundedness properties of non-linear schemes across discontinuities. Therefore, this study proposes a diagnostic tool to evaluate and improve the convection boundedness of numerical schemes across discontinuities.

This work is organised as follows. In Section \ref{sec:CBC}, the convection boundedness criterion across discontinuities and the proposed diagnostic method based on the normalised variable diagram are given. In Section \ref{sec:evaluation}, we apply the proposed method to evaluate the representative schemes and validate the conclusions drawn through numerical tests. In Section \ref{sec:improvement}, we demonstrate the application of the proposed method by formulating a new THINC scheme with less stringent CFL conditions. Finally, a brief concluding remark is given in Section \ref{sec:conclusion}.

\section{Convection boundedness criterion across discontinuities} \label{sec:CBC}
We consider the one-dimensional linear advection equation with constant advection velocity over the domain $\Omega$
\begin{equation} \label{eq:control}
\left\{
\begin{array}{ll}
\frac{\partial \phi}{\partial t} + u\frac{\partial \phi }{ \partial x} = 0~~\forall x \in \Omega, \\
\phi(x,t=0)=\phi_0(x),
\end{array}
\right.   
\end{equation}
where $\phi$ is the solution variable and $u$ is the (constant) advection velocity and is assumed $u>0$ here. To solve the above equation with the finite volume method, we define the discrete finite volume cell $\mathcal{I}_i\equiv[x_{i-\frac{1}{2}},x_{i+\frac{1}{2}}]$ and divide the computational domain $\Omega$ into $N$ non-overlapping uniform finite volume cell $\Omega=\cup_{i=1}^N \mathcal{I}_i$ with uniform cell size $h$. We define the discrete cell average value $\bar{\phi}_i$ over cell $\mathcal{I}_i$ as
\begin{equation}
    \bar{\phi}_i (t) = \frac{1}{h} \int_{x_{i-\frac{1}{2}}}^{x_{i+\frac{1}{2}}} \phi(x,t) dx.
\end{equation}
Then the cell average of each cell, $\bar{\phi}_i$, is updated by the following semi-discrete form using the finite volume method
\begin{equation} \label{semi_form_1}
\frac{d \bar{\phi}_i}{d t} = -\frac{u}{h} (\phi_{i+\frac{1}{2}}-\phi_{i-\frac{1}{2}}), 
\end{equation}
where $\phi_{i+\frac{1}{2}}$ and $\phi_{i-\frac{1}{2}}$ are reconstructed value at cell boundary $i+\frac{1}{2}$ and $i-\frac{1}{2}$, respectively. The above ordinary differential equation can be solved using the high-order Runge-Kutta (RK) time scheme. As the first-order explicit Euler time scheme is the building block of the high-order RK schemes, we will consider the first-order explicit Euler time scheme with a constant time step $\Delta t$ to update the solution here. Then the discretized form of Eq.~(\ref{eq:control}) is 
\begin{equation} \label{eq:discreteForm}
  \bar{\phi}_i^{n+1}= \bar{\phi}_i^{n}-c(\phi_{i+\frac{1}{2}}-\phi_{i-\frac{1}{2}}), 
\end{equation}
where $c=\frac{u \Delta t}{h}$ is the CFL number. $\bar{\phi}_i^{n}$ and $\bar{\phi}_i^{n+1}$ are solution at $t^n$ and $t^{n+1}$.

The process to obtain the reconstructed value at the cell boundary from cell average values is well-known as the reconstruction process: $\phi_{i+\frac{1}{2}}=\mathcal{R}(\dotsb,\bar{\phi}_{i-1}^n,\bar{\phi}_{i}^n, \bar{\phi}_{i+1}^n,\dotsb)$. The reconstruction scheme, $\mathcal{R}$, is designed to satisfy several desirable properties depending on the distribution of cell average values $\{ \dotsb,\bar{\phi}_{i-1}^n,\bar{\phi}_{i}^n, \bar{\phi}_{i+1}^n,\dotsb \}$. When the distribution of cell average values is smooth, $\mathcal{R}$ is usually designed to achieve certain accuracy of order $p$ as $(\phi_{i+\frac{1}{2}}^e-\phi_{i+\frac{1}{2}})=O(h^p)$, where $\phi_{i+\frac{1}{2}}^e$ is the exact value. However, when the distribution contains discontinuities, the convection boundedness property is essential for the reconstruction scheme to capture the discontinuity. According to Godunov's theorem, designing non-linear reconstruction schemes is required to simultaneously satisfy the accuracy condition ($p>1$) and convection boundedness. Although it is well-established to design numerical schemes of high order, the analysis tool for the convection boundedness of non-linear schemes across discontinuities is not yet available. This work aims to fill this gap. We first give the following definitions.

\begin{definition} \label{def:discontinuity}
Cell $i$ contains a perfect isolated discontinuity if $\dotsb=\bar{\phi}_{i-2}^n=\bar{\phi}_{i-1}^n\leq\bar{\phi}_{i}^n<\bar{\phi}_{i+1}^n=\bar{\phi}_{i+2}^n=\dotsb$.   
\end{definition}
It is noted that the discontinuity cannot be well defined by only three cells as pointed out by the work \cite{liu2019discontinuity}. The isolated discontinuity can also be defined by $\dotsb=\bar{\phi}_{i-2}^n=\bar{\phi}_{i-1}^n\geq\bar{\phi}_{i}^n>\bar{\phi}_{i+1}^n=\bar{\phi}_{i+2}^n=\dotsb$. For simplicity, we will only consider the first situation here and the conclusion drawn will remain the same.     

\begin{definition}
For the cell $i$ containing a perfect isolated discontinuity and $0<c<1$, the numerical scheme satisfies the convection boundedness criterion across discontinuities if $\dotsb=\bar{\phi}_{i-2}^{n+1}=\bar{\phi}_{i-1}^{n+1}\leq\bar{\phi}_{i}^{n+1}<\bar{\phi}_{i+1}^{n+1}\leq\bar{\phi}_{i+2}^{n+1}=\dotsb$, $\bar{\phi}_{i-2}^{n+1}=\bar{\phi}_{i-2}^{n}$, and $\bar{\phi}_{i+2}^{n+1}=\bar{\phi}_{i+2}^{n}$.       
\end{definition}
It should be noted that the convection boundedness criterion prohibits any overshoot or undershoot across discontinuities, but it permits the discontinuity to be diffused, i.e. $\bar{\phi}_{i-1}^{n+1}\leq\bar{\phi}_{i}^{n+1}<\bar{\phi}_{i+1}^{n+1}\leq\bar{\phi}_{i+2}^{n+1}$. 

\begin{definition}
The normalised cell average value $\tilde{\phi}_{i}$ is defined as $\tilde{\phi}_{i}=\frac{\bar{\phi}_{i}-\bar{\phi}_{i-1}}{\bar{\phi}_{i+1}-\bar{\phi}_{i-1}}$. The normalised reconstructed variable $\tilde{\phi}_{i+1/2}$ is defined as $\tilde{\phi}_{i+1/2}=\frac{\phi_{i+1/2}-\bar{\phi}_{i-1}}{\bar{\phi}_{i+1}-\bar{\phi}_{i-1}}$    
\end{definition}

\begin{proposition}
The necessary conditions for the numerical schemes to satisfy the convection boundedness criterion across the discontinuities are $\tilde{\phi}_{i+1/2}\leq \frac{1}{c} \tilde{\phi}_{i}$ and $\tilde{\phi}_{i+1/2}\leq 1$.     
\end{proposition}

\begin{proof}
By convection boundedness, $\bar{\phi}_{i-2}^{n+1}=\bar{\phi}_{i-1}^{n+1}=\bar{\phi}_{i-1}^{n}$,  and it then follows from 
Eq. (\ref{eq:discreteForm}) that we have $\phi_{i-3/2}^n=\phi_{i-1/2}^n=\bar{\phi}_{i-1}^{n}$; the latter equality also follows from convection boundedness. Then,
\begin{equation*}
    \bar{\phi}_{i-1}^{n+1}\leq\bar{\phi}_{i}^{n+1},~\Rightarrow ~ \bar{\phi}_{i-1}^{n}\leq c \bar{\phi}_{i-1}^n-c\phi_{i+1/2}^n +\bar{\phi}_{i}^n,~\Rightarrow ~ c \frac{\phi_{i+1/2}^n-\bar{\phi}_{i-1}^n}{\bar{\phi}_{i+1}^n-\bar{\phi}_{i-1}^n}\leq\frac{\bar{\phi}_{i+1}^n-\bar{\phi}_{i-1}^n}{\bar{\phi}_{i+1}^n-\bar{\phi}_{i-1}^n},~\Rightarrow ~ \tilde{\phi}_{i+1/2}^n\leq \frac{1}{c} \tilde{\phi}_{i}^n.     
\end{equation*}
Similarly, as $\bar{\phi}_{i+2}^{n+1}=\bar{\phi}_{i+3}^{n+1}=\bar{\phi}_{i+1}^{n}$. we have $\phi_{i+3/2}^n=\phi_{i+1/2}^n=\bar{\phi}_{i+1}^{n}$. Then
\begin{equation*}
 \bar{\phi}_{i+1}^{n+1}\leq\bar{\phi}_{i+2}^{n+1},~\Rightarrow ~ c \phi_{i+1/2}^n-c \bar{\phi}_{i+1}^{n} + \bar{\phi}_{i+1}^{n}\leq \bar{\phi}_{i+1}^{n},~\Rightarrow ~ \tilde{\phi}_{i+1/2}^n\leq 1.    
\end{equation*}

\end{proof}
The necessary condition of $\tilde{\phi}_{i+1/2}\leq \frac{1}{c} \tilde{\phi}_{i}$ determines the maximum CFL number for a numerical scheme to satisfy the convection boundedness criterion. Furthermore, when a numerical scheme violates the convection boundedness criterion, the location of the resulting overshoot (or undershoot) depends on the specific condition that has been violated. We will illustrate these remarks in the following sections. 

Based on the definitions of discontinuity and the convection boundedness criterion, the key to evaluating the convection boundedness of the scheme lies in identifying the normalised reconstructed variable, $\tilde{\phi}_{i+1/2}$, for the corresponding schemes where $0 < \tilde{\phi}_i < 1$. To achieve this, we propose a diagnostic tool by using the normalised variable diagram which shows the variation of $\tilde{\phi}_{i+1/2}$ with $\tilde{\phi}_i$ for different schemes.

\section{Evaluation of representative shock-capturing schemes} \label{sec:evaluation}
In this section, we will examine the characteristics of representative non-linear shock-capturing schemes across discontinuities using the above convection boundedness criterion. The first scheme we examine here is the THINC scheme. As shown by the work \cite{deng2023unified}, the normalised reconstructed variable of the THINC scheme across the discontinuity is
\begin{equation}
  \tilde{\phi}_{i+1/2}= \dfrac{\sinh(\beta)+\cosh(\beta)-\text{exp}(\beta(-2.0\tilde{\phi}_i+1.0))}{2\sinh(\beta)},  
\end{equation}
where $\beta$ is the parameter to control the numerical errors. A higher value of $\beta$ introduces greater numerical anti-diffusion. The normalised variable diagram for the THINC scheme across the discontinuity with different $\beta$ values is presented in Fig.~\ref{fig:THINCCFL}. It is observed that as $\beta$ increases, the maximum CFL number allowed for the THINC scheme to satisfy the convection boundedness criterion decreases. For instance, when $\beta = 1.1$, the maximum CFL number must not exceed 0.5. Similarly, for $\beta = 2.0$, the maximum CFL number is limited to 0.3. 

We further perform an advection test of a square wave to validate the proposed normalised variable diagram. The computational domain is $[-1,1]$. The square wave is initialised as 
\begin{equation} \label{eq:initial}
\phi(x,t=0)=\left\{
\begin{array}{ll}
1.0,~~-0.4 \leq x \leq 0.4, \\
0.0,~~\text{otherwise}.
\end{array}
\right.   
\end{equation}
The computation is run up to $t=0.1$. The numerical solutions using the THINC scheme with different $\beta$ values at various CFL numbers are presented in Fig.~\ref{fig:advectionTHINC}. It can be observed that the THINC scheme with $\beta=1.1$ at CFL=0.5 and the THINC scheme with $\beta=2.0$ at CFL=0.3 produce over- and under-shoots, consistent with the conclusions drawn from the normalised variable diagram. Furthermore, the over- and under-shoots occur at $\tilde{\phi}_i<0.5 $, where the condition of $\tilde{\phi}_{i+1/2}\leq \frac{1}{c} \tilde{\phi}_{i}$ is violated.


\begin{figure}[htbp] 
\begin{centering}
\subfigure[Global view]{\includegraphics[scale=0.4,trim={0.3cm 0.5cm 0.5cm 0.3cm},clip]{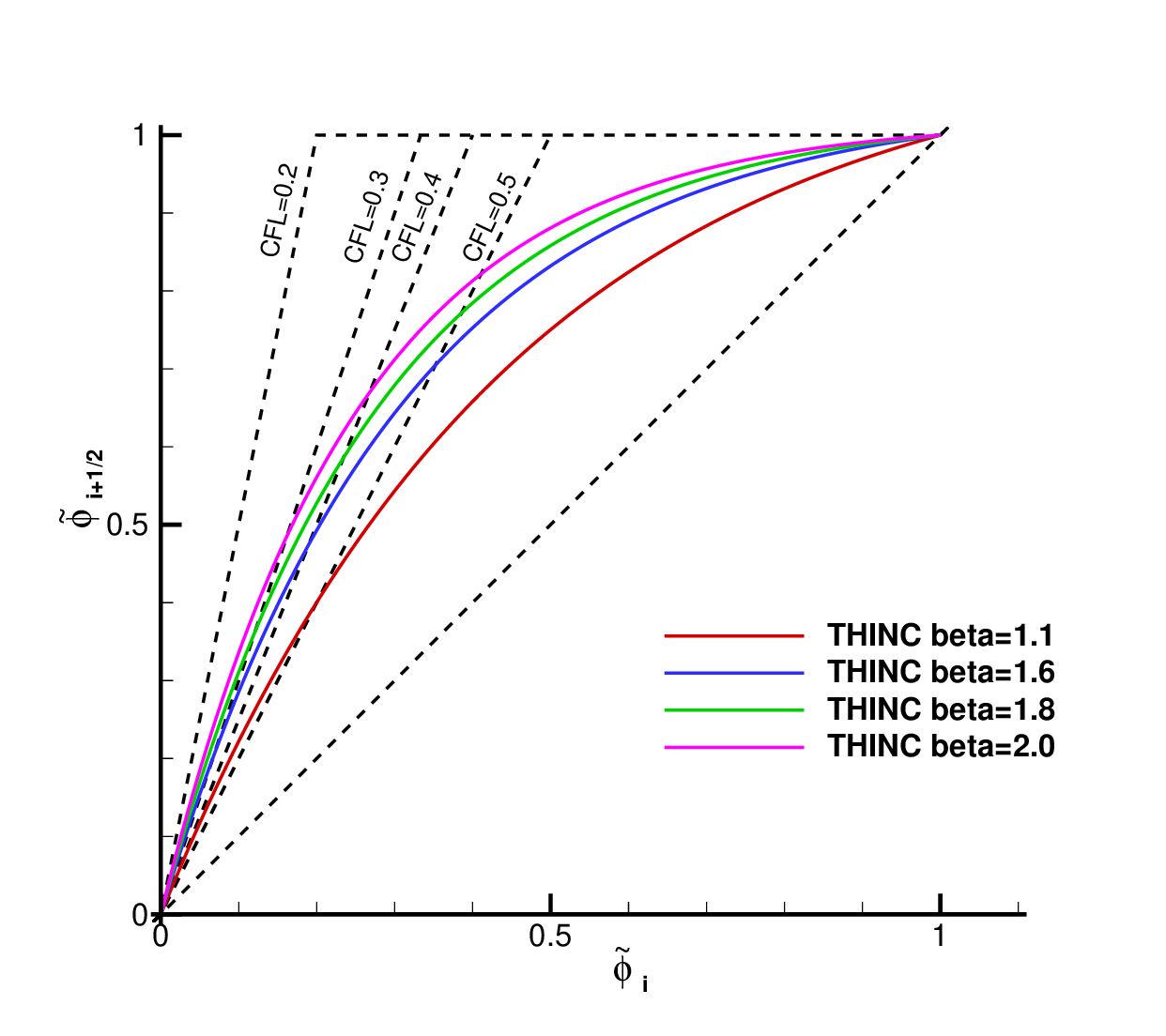}}
\subfigure[Zoomed view ]{\includegraphics[scale=0.4,trim={0.3cm 0.5cm 0.5cm 0.3cm},clip]{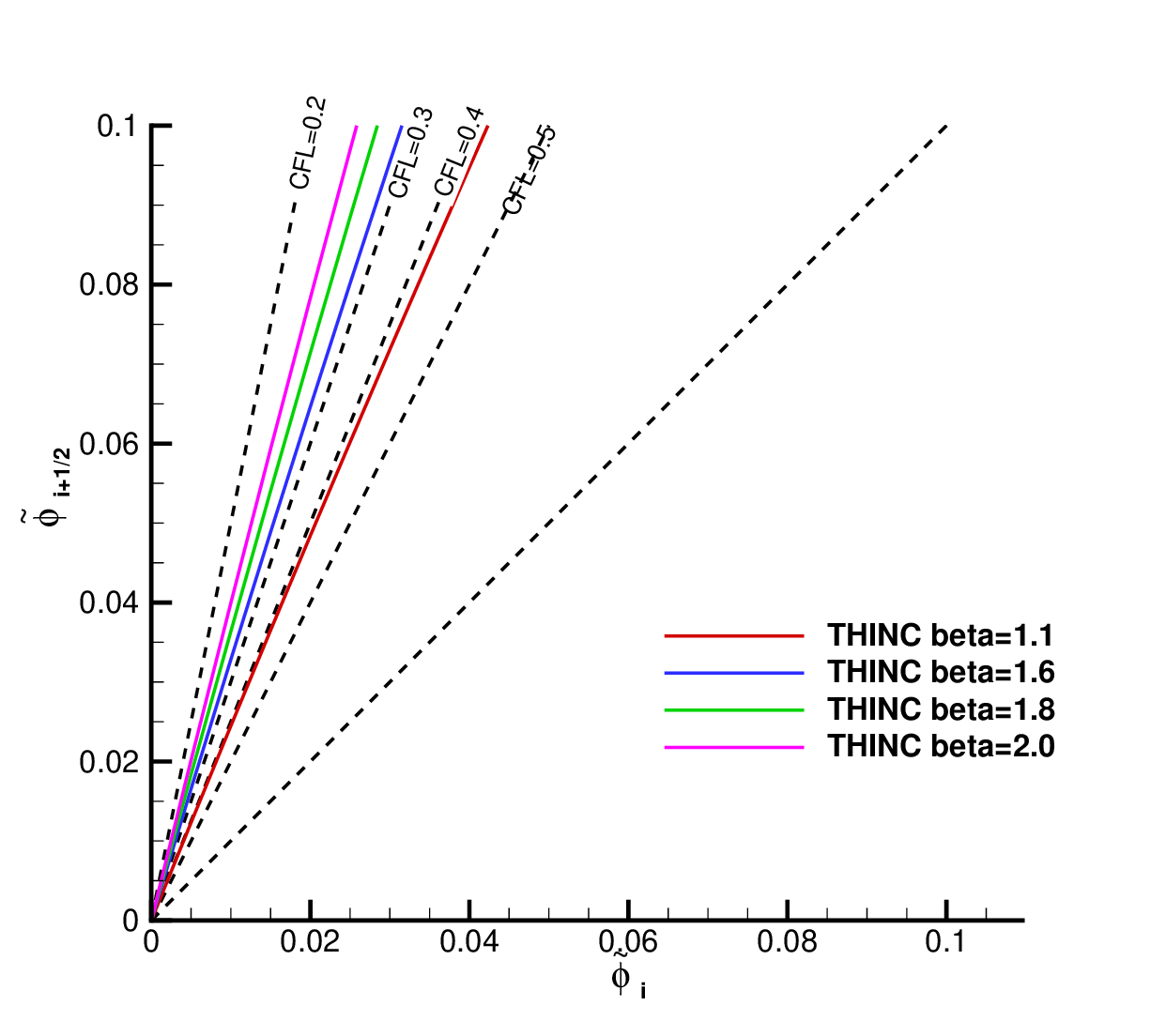}}
\caption{The variation of $\tilde{\phi}_{i+1/2}$ with $\tilde{\phi}_i$ for the THINC scheme across the discontinuity with different $\beta$ value. \label{fig:THINCCFL}}
\par\end{centering}
\end{figure}

\begin{figure}[htbp] 
\begin{centering}
\subfigure[THINC $\beta=1.1$, CFL=0.4 and 0.5]{\includegraphics[scale=0.4,trim={0.5cm 0.5cm 0.5cm 0.3cm},clip]{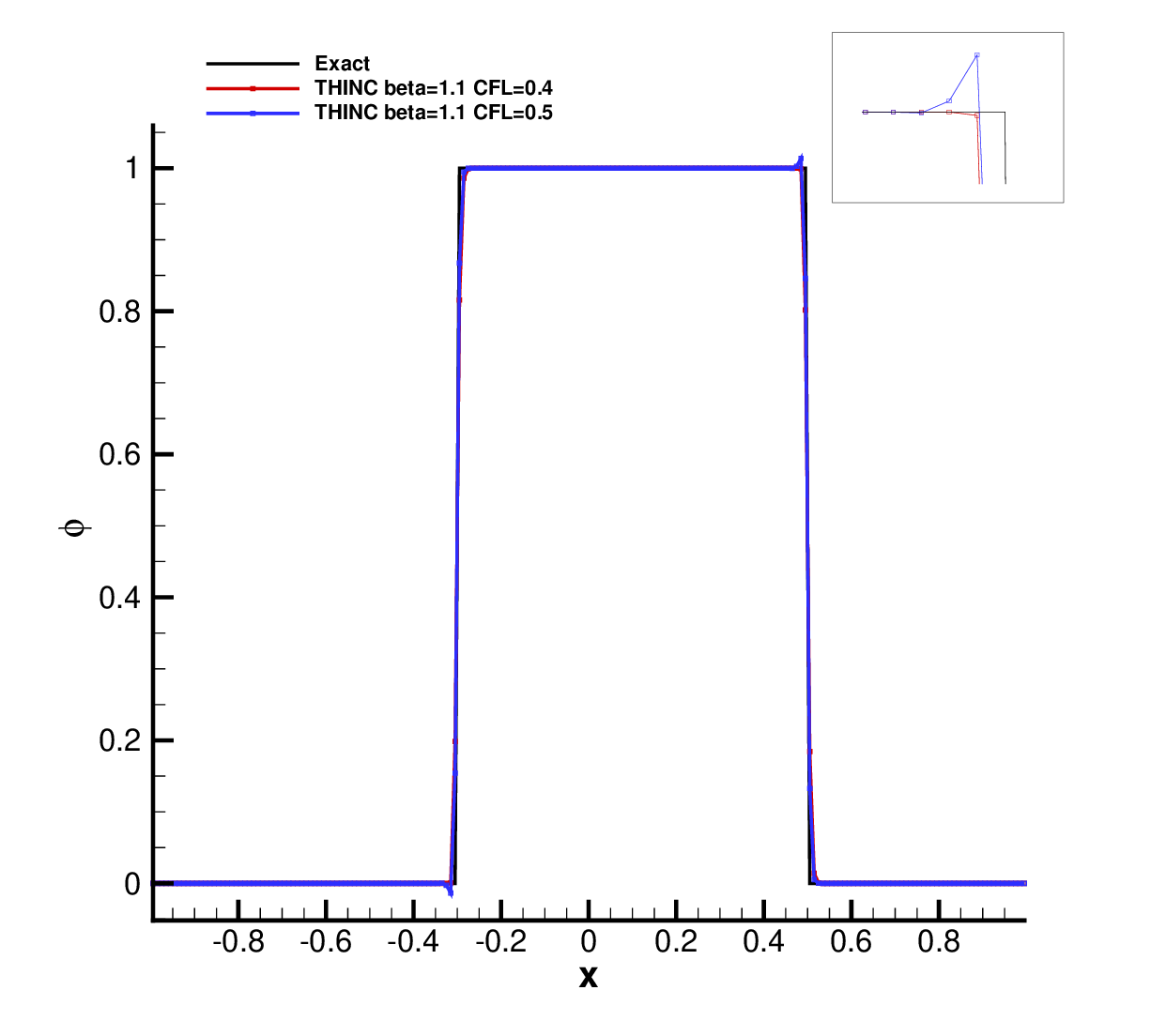}}
\subfigure[THINC $\beta=2.0$, CFL=0.2 and 0.3]{\includegraphics[scale=0.4,trim={0.5cm 0.5cm 0.5cm 0.3cm},clip]{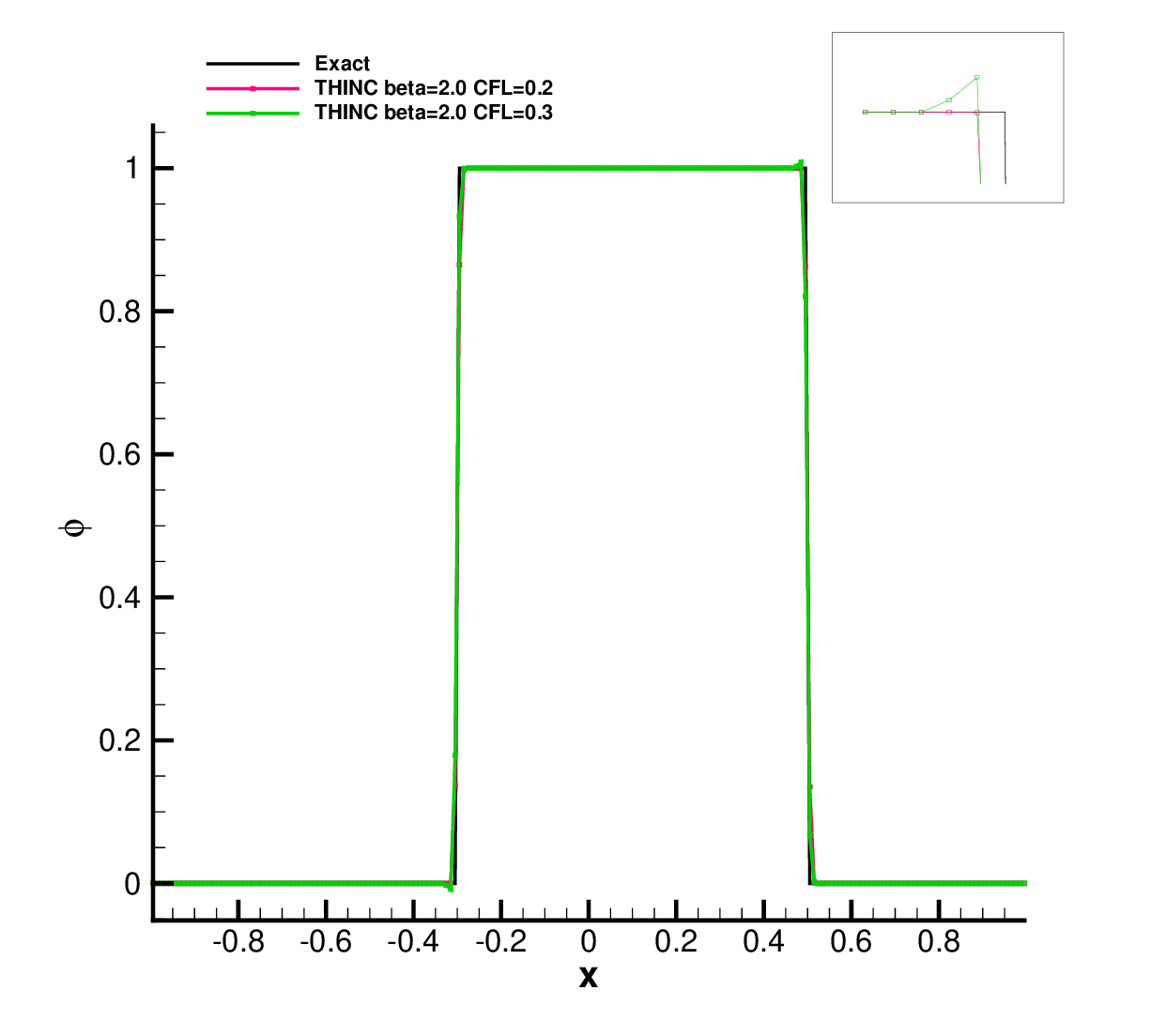}}
\caption{Numerical solutions for the advection of a square wave using the THINC scheme with different $\beta$ values at various CFL numbers.  \label{fig:advectionTHINC}}
\par\end{centering}
\end{figure}

Furthermore, we evaluate the widely used shock-capturing schemes, including WENO and TENO. The reconstructed variable of the WENO and TENO scheme is a non-linear combination of reconstructed values from different stencils, i.e. 
\begin{equation}
    \phi_{i+1/2}=\sum_{k=0}^{2} \omega_k \phi_{i+1/2,k}=\omega_0(\frac{1}{3}\bar{\phi}_{i-2}-\frac{7}{6}\bar{\phi}_{i-1}+\frac{11}{6}\bar{\phi}_{i})+\omega_1(-\frac{1}{6}\bar{\phi}_{i-1}+\frac{5}{6}\bar{\phi}_{i}+\frac{1}{3}\bar{\phi}_{i+1})+ \omega_2(\frac{1}{3}\bar{\phi}_{i}+\frac{5}{6}\bar{\phi}_{i+1}-\frac{1}{6}\bar{\phi}_{i+2}),
\end{equation}
where $\omega_k$ are the non-linear weights and $\phi_{i+1/2,k}$ are the reconstructed value from different stencils. In \cite{jiang96}, the non-linear weights are defined as 
\begin{equation}
    \omega_k=\frac{\alpha_k}{\sum_{k=0}^2 \alpha_k},~~\alpha_k=\frac{d_k}{(IS_k+\epsilon)^2},
\end{equation}
where $d_k$ are ideal weights, $IS_k$ are local smoothness indicators, and $\epsilon$ is a small value to prevent division by zero. The resulting WENO scheme is denoted as WENOJS5. The work \cite{wenoZ_Borges2008} designs another weight as
\begin{equation}
    \alpha_k=d_k\Bigl(1+(\frac{\tau_5}{IS_k+\epsilon})^2\Bigl),~~\tau_5=|IS_0-IS_2|.
\end{equation}
The resulting WENO scheme with the weight mentioned above is typically referred to as WENOZ5. In \cite{fu2016family}, a new fifth-order low-dissipation weighting strategy named TENO5 is developed. The nonlinear weights for the TENO5 scheme are formulated as
\begin{equation}
\omega_k=\frac{d_k \delta_k}{\sum_{k=0}^2 d_k \delta_k},~~\delta_k=\left\{
\begin{array}{ll}
0,~~\text{if}~\chi_k<C_T, \\
1,~~\text{otherwise},
\end{array}
\right. ~~\chi_k=\frac{\gamma_k}{\sum_{k=0}^2 \gamma_k},~~\gamma_k=\Bigl(C+\frac{\tau_5}{IS_k+\epsilon}\Bigl)^q,
\end{equation}
where parameters $C=1$ and $q=6$ are used in the work \cite{fu2016family}. $C_T$ is the cutoff parameter. The influence of $C_T$ has been studied in previous work \cite{fu2016family}, showing that a smaller $C_T$ results in less dissipative results. In this work, we will, for the first time, demonstrate how $C_T$ influences the convection bounded criterion of the TENO5 scheme across discontinuity.  

Unlike the THINC scheme, the $\tilde{\phi}_i$–$\tilde{\phi}_{i+1/2}$ relationship in the WENO and TENO scheme cannot be directly expressed in a closed form. Here, we make the following assumption.  
\begin{assumption}
    We assume that the WENOJS5, WENOZ5 and TENO5 schemes are scale-independent.
\end{assumption}
Such an assumption is reasonable within the range where the smoothness indicator $IS_k\gg\epsilon$. Using this assumption and Definition \ref{def:discontinuity}, we can generate the normalised variable diagram for the WENO and TENO schemes across the discontinuity by setting $\bar{\phi}_{i-2}=\bar{\phi}_{i-1}=0$ and $\bar{\phi}_{i+1}=\bar{\phi}_{i+2}=1$, while sampling $\bar{\phi}_i$ between 0 and 1. The normalised variable diagram for the WENOJS5 and WENOZ5 schemes across the discontinuity is presented in Fig.~\ref{fig:WENOTENOCFL}(a), with $\bar{\phi}_i$ uniformly sampled between 0 and 1 at 100 points. Fig.~\ref{fig:WENOTENOCFL}(a) indicates that WENOZ5 requires a stricter CFL number (maximum CFL$\approx0.4$) to satisfy the convection boundedness condition compared to WENOJS5. Similarly, we generate the normalized variable diagram for the TENO5 scheme across the discontinuity using two different values of the parameter, i.e., $C_T = 10^{-5}$ and $C_T = 10^{-7}$. The generated diagram is presented in Fig.~\ref{fig:WENOTENOCFL}(b). The diagram shows that TENO5 using $C_T = 10^{-7}$ requires a stricter CFL number to satisfy $\tilde{\phi}_{i+1/2}\leq \frac{1}{c} \tilde{\phi}_{i}$ than the scheme using $C_T = 10^{-5}$. Moreover, TENO5 using $C_T = 10^{-7}$ violates the condition $\tilde{\phi}_{i+1/2}\leq 1$, indicating that the overshoots cannot be eliminated by reducing the CFL number. Thus, $C_T = 10^{-7}$ is unsuitable for the fifth-order TENO scheme in capturing discontinuities. 

\begin{figure}[htbp] 
\begin{centering}
\subfigure[WENOJS5 and WENOZ5]{\includegraphics[scale=0.4,trim={0.5cm 0.5cm 0.5cm 0.3cm},clip]{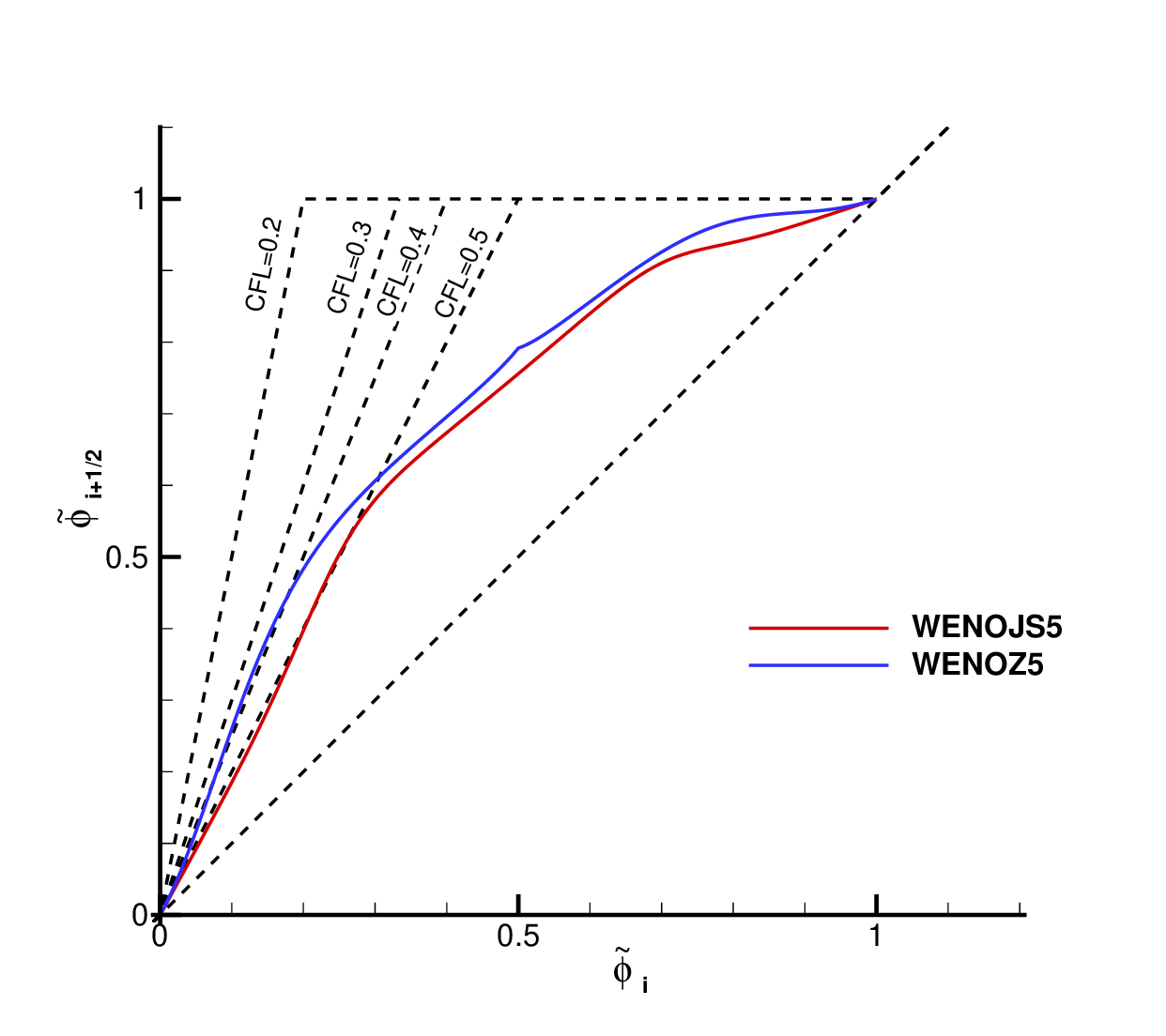}}
\subfigure[TENO5 using $C_T = 10^{-5}$ and $C_T = 10^{-7}$ ]{\includegraphics[scale=0.4,trim={0.5cm 0.5cm 0.5cm 0.3cm},clip]{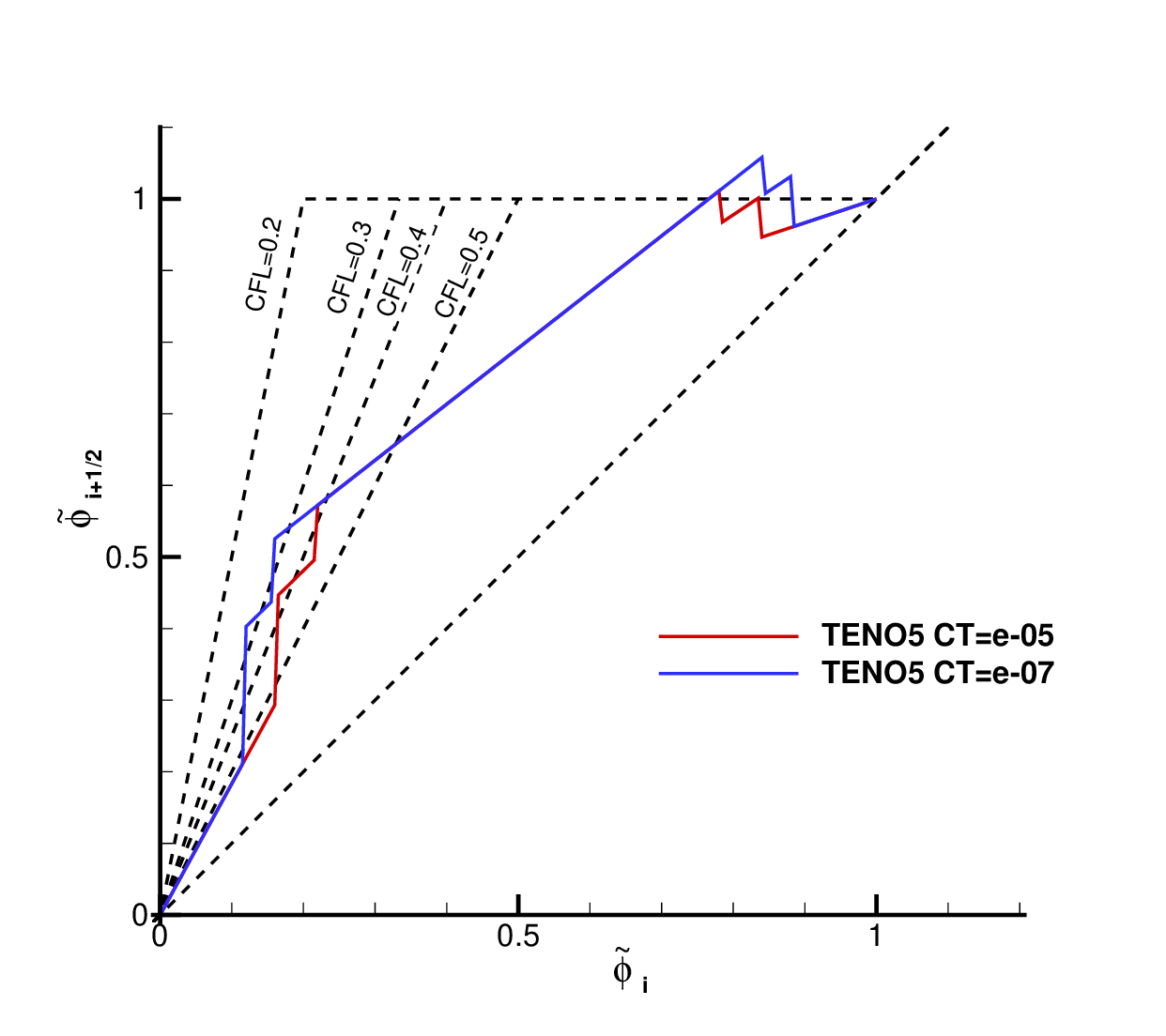}}
\caption{The normalised variable diagram for the WENOJS5, WENOZ5 and TENO5 schemes across the discontinuity.  \label{fig:WENOTENOCFL}}
\par\end{centering}
\end{figure}

An advection test of a square wave is performed to illustrate the conclusion drawn from the normalised variable diagram. Fig.~\ref{fig:advectionWT}(a) shows that WENOZ5 requires a stricter CFL number (maximum CFL$\approx0.4$) than WENOJS5 to prevent over- and under-shoots. This observation is consistent with the conclusion drawn from the normalized variable diagram. In Fig.~\ref{fig:advectionWT}(b), we show the results produced by the TENO schemes using $C_T=10^{-5}$ and $C_T=10^{-7}$. It can be observed that TENO using $C_T=10^{-5}$ is convection bounded at CFL=0.5, while TENO using $C_T=10^{-7}$ produces over- and under-shoots at CFL=0.4. These over- and under-shoots occur both when $\tilde{\phi}_{i}$ is close to 0 and when $\tilde{\phi}_{i}$ is close to 1. We reduce the CFL number to 0.1 for the TENO scheme using $C_T=10^{-7}$ and show the results in Fig.~\ref{fig:advectionWT}(c). It can be observed that over- and under-shoots when $\tilde{\phi}_{i}$ is close to 0 can be eliminated by reducing the CFL number, as the condition $\tilde{\phi}_{i+1/2}\leq \frac{1}{c} \tilde{\phi}_{i}$ is satisfied with a smaller CFL number. However, over- and under-shoots when $\tilde{\phi}_{i}$ is close to 1 cannot be eliminated, as TENO5 using $C_T=10^{-7}$ violates the condition $\tilde{\phi}_{i+1/2}\leq 1$ even with a smaller CFL number. This observation about the TENO5 scheme is consistent with the conclusion from the normalised variable diagram.        

\begin{figure}[htbp] 
\begin{centering}
\subfigure[WENOJS5 and WENOZ5 at various CFL number]{\includegraphics[scale=0.4,trim={0.5cm 0.5cm 0.5cm 0.3cm},clip]{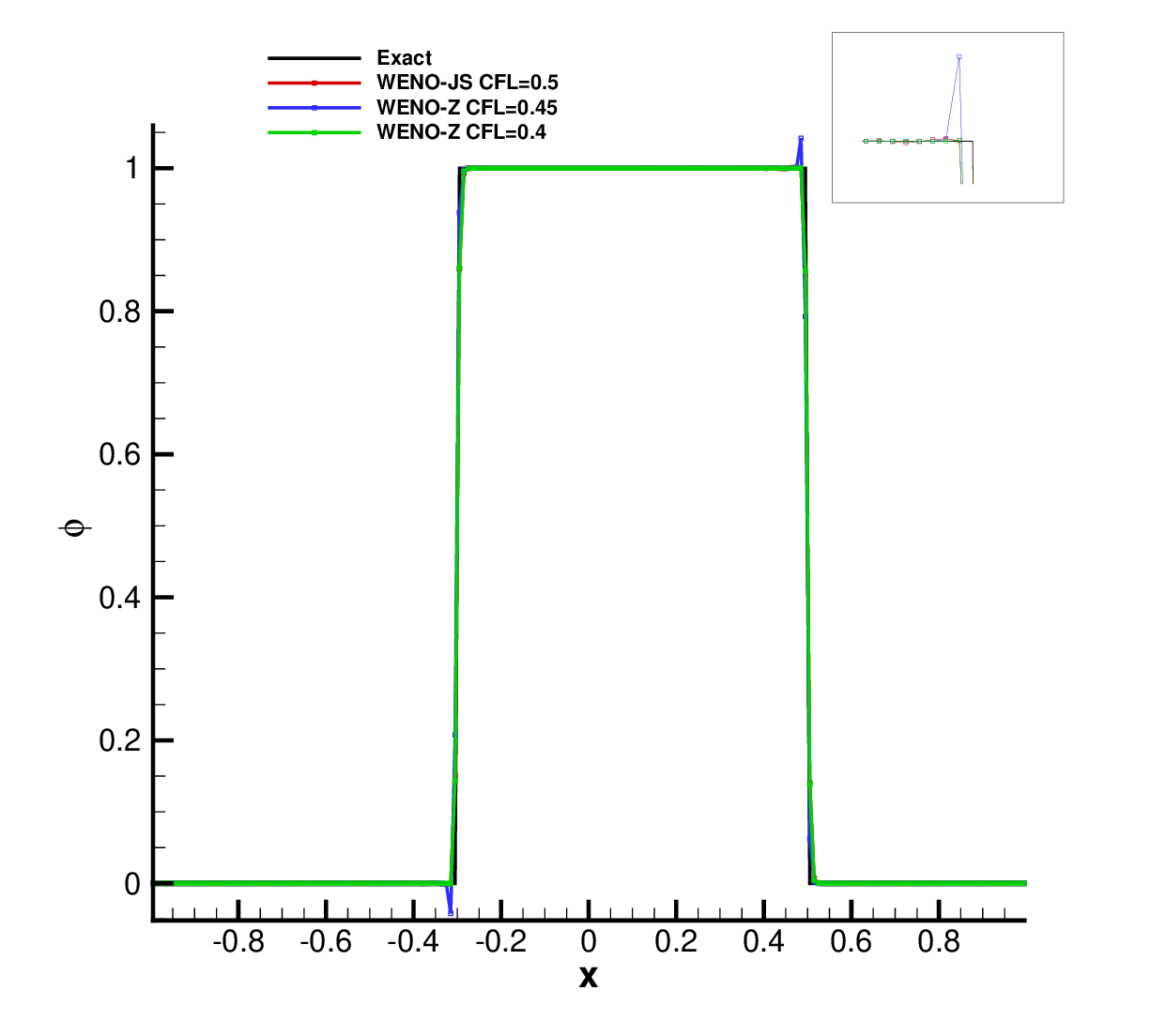}}
\subfigure[TENO5 using $C_T=10^{-5}$ and $C_T=10^{-7}$ ]{\includegraphics[scale=0.4,trim={0.5cm 0.5cm 0.5cm 0.3cm},clip]{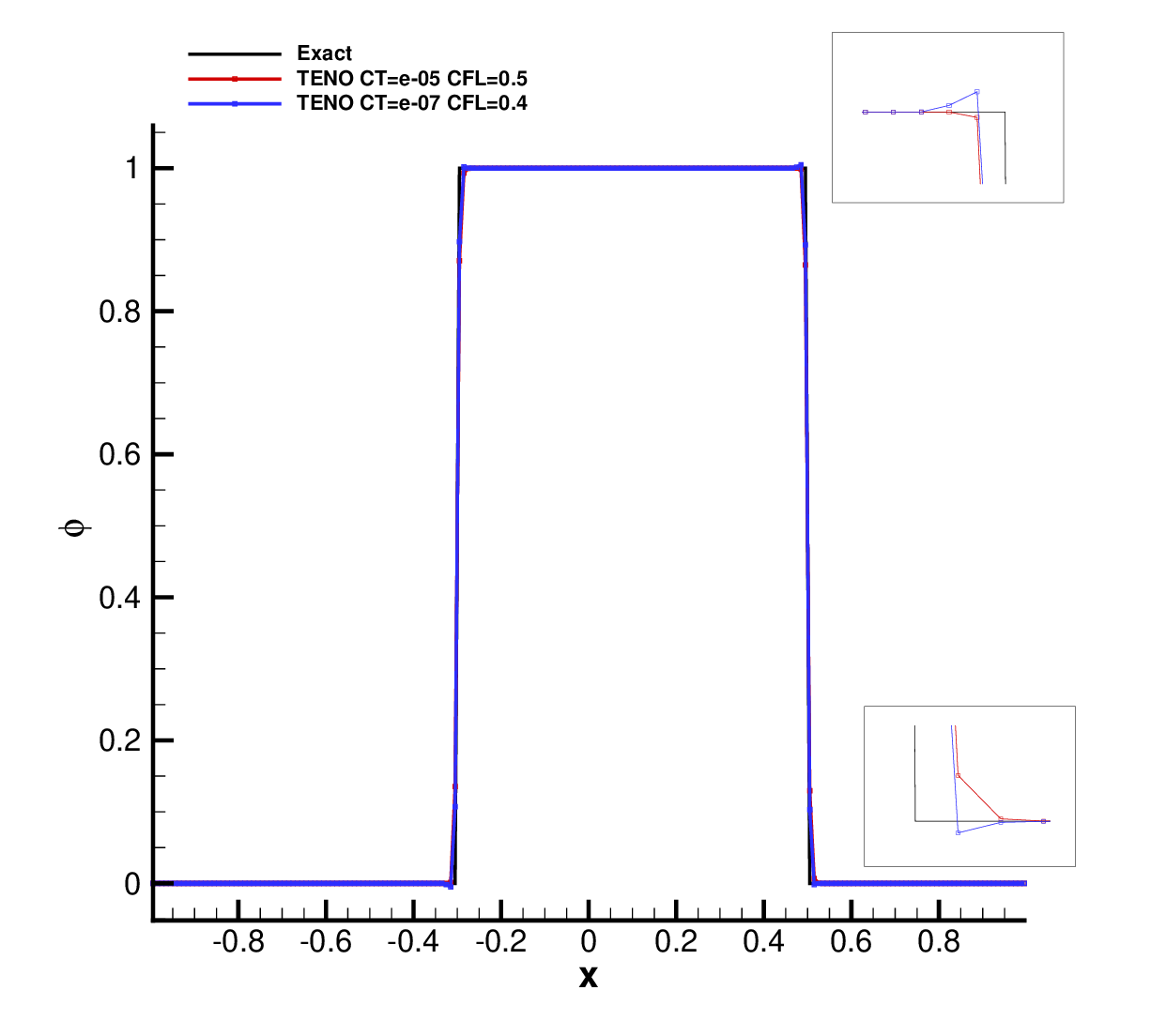}}
\subfigure[TENO5 using $C_T=10^{-7}$ at CFL=0.4 and CFL=0.1 ]{\includegraphics[scale=0.4,trim={0.5cm 0.5cm 0.5cm 0.3cm},clip]{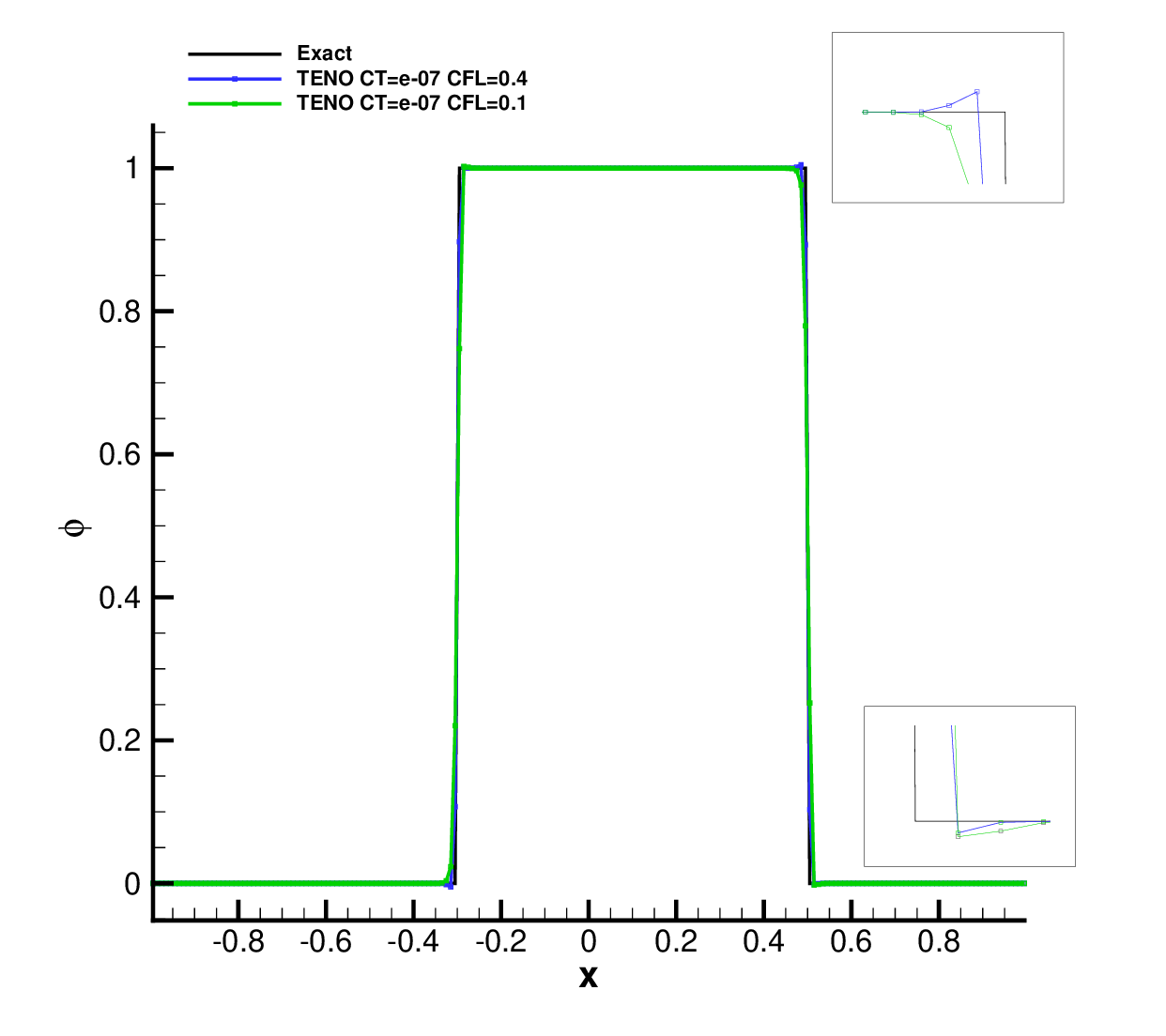}}
\caption{Numerical solutions for the advection of a square wave using the WENO and TENO scheme with different parameters at various CFL numbers.  \label{fig:advectionWT}}
\par\end{centering}
\end{figure}

\section{Improvement of convection boundedness of THINC schemes} \label{sec:improvement}
In this section, we will demonstrate how to improve the convection boundedness property of the THINC scheme by using the proposed normalised variable diagram. The THINC scheme which is less constrained by the CFL condition is formulated as
\begin{equation}
   \tilde{\phi}_{i+1/2}=\text{min} \Bigl( \dfrac{\sinh(\beta)+\cosh(\beta)-\text{exp}(\beta(-2.0\tilde{\phi}_i+1.0))}{2\sinh(\beta)}, 2.5\tilde{\phi}_i\Bigl).    
\end{equation}
According to the convection boundedness criterion, the new formulation allows the maximum CFL=0.4. We conduct the 1D advection test of a square wave and 2D Zalesak test \cite{zalesak1979fully,rudman1997volume} using the original and new THINC schemes with $\beta=2.0$. The numerical solutions of the 1D advection test are presented in Fig.~\ref{fig:advectionNewT}(a). It is observed that the original THINC scheme produces over- and under-shoots at CFL=0.4, while the new THINC scheme remains bounded. The numerical solutions of the 2D Zalesak test using the original and new THINC schemes are presented in Fig.~\ref{fig:advectionNewT}(b) and Fig.~\ref{fig:advectionNewT}(c), respectively. The results indicate that the new THINC scheme can capture sharp discontinuity while remaining bounded, whereas the original THINC scheme produces over- and under-shoots.    

\begin{figure}[htbp] 
\begin{centering}
\subfigure[Original and new THINC scheme, $\beta=2.0$, CFL=0.4]{\includegraphics[scale=0.4,trim={0.5cm 0.5cm 0.5cm 0.3cm},clip]{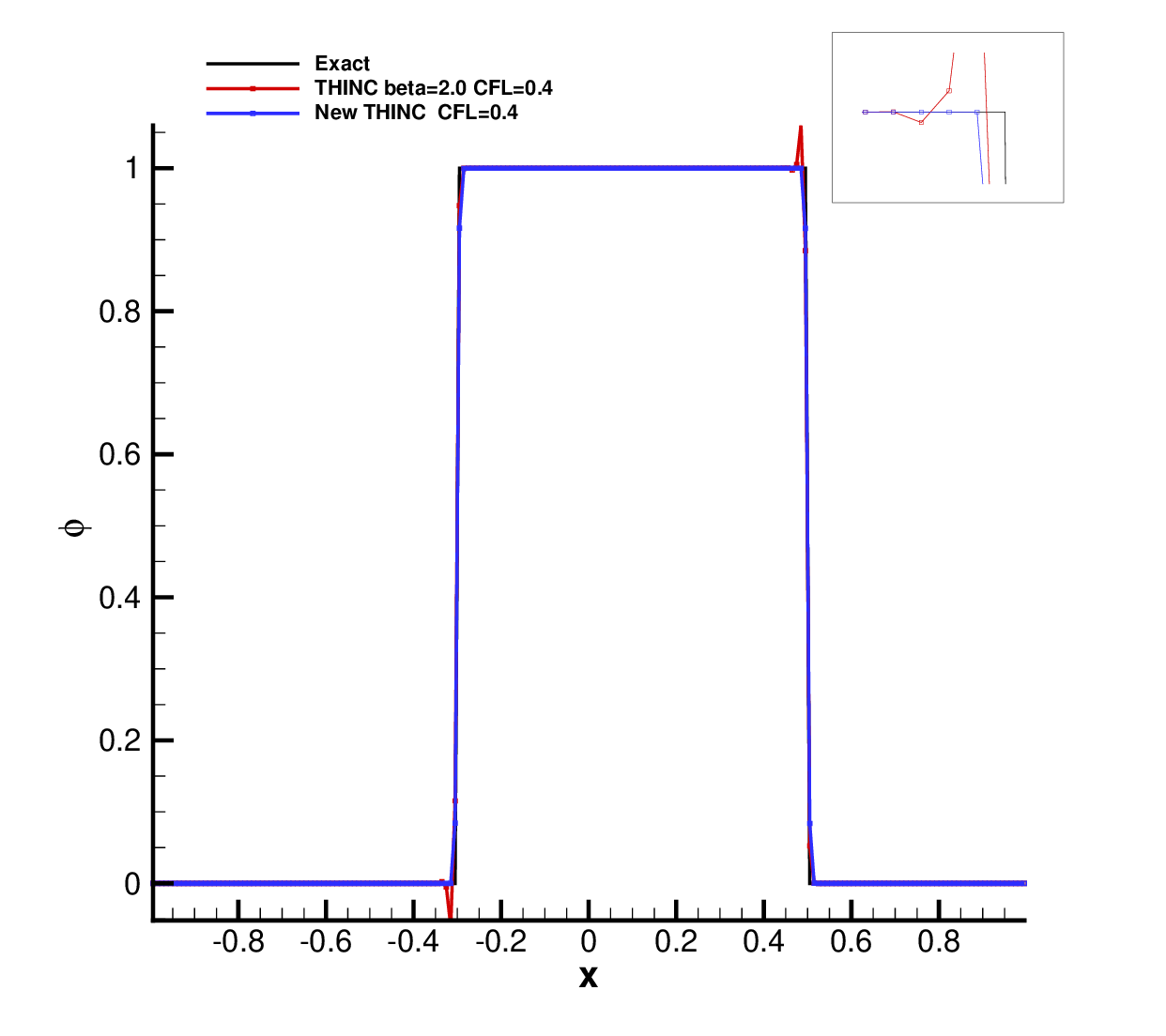}}
\subfigure[Original THINC ]{\includegraphics[scale=0.4,trim={0.5cm 0.5cm 0.5cm 0.3cm},clip]{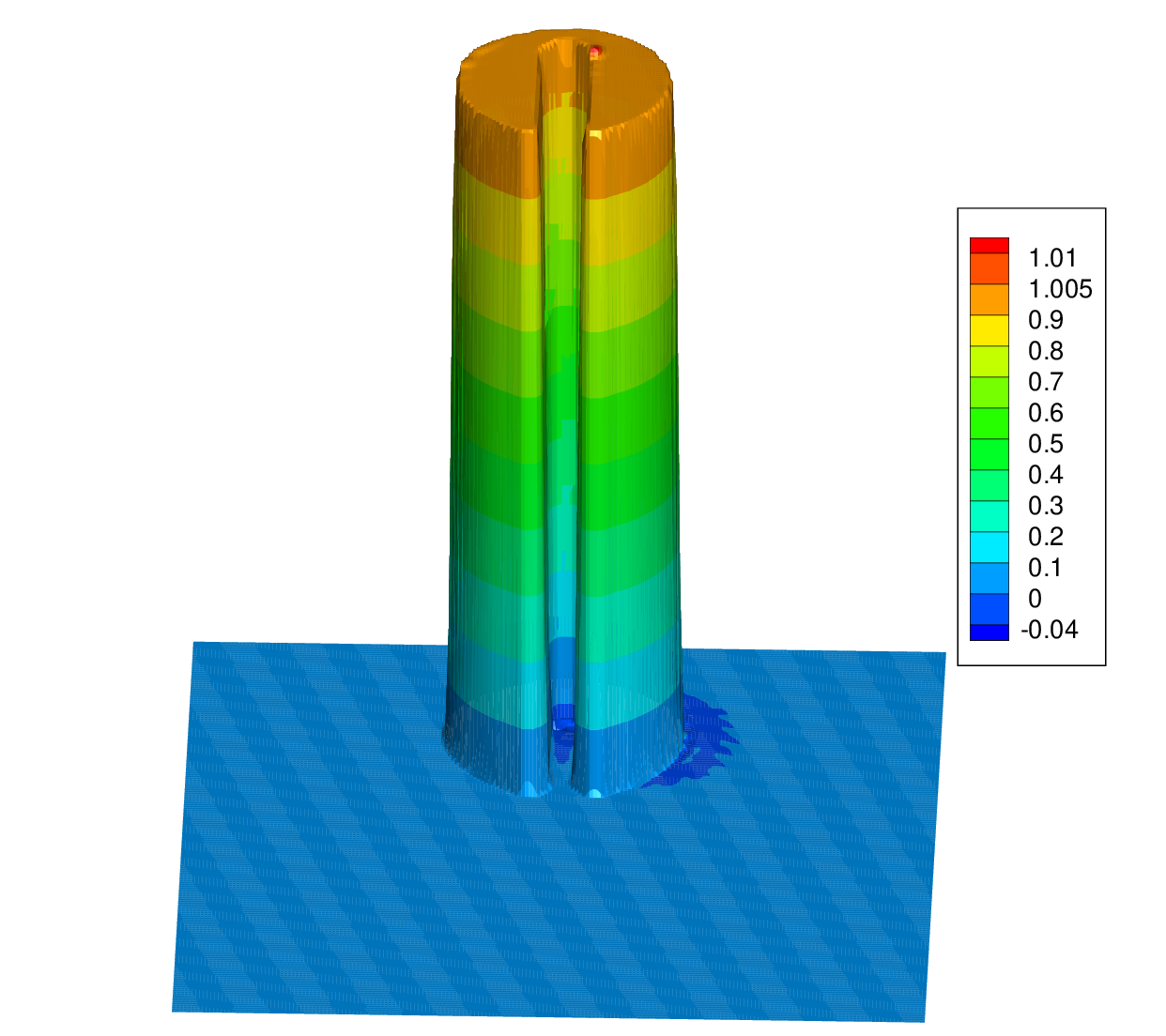}}
\subfigure[New THINC]{\includegraphics[scale=0.4,trim={0.5cm 0.5cm 0.5cm 0.3cm},clip]{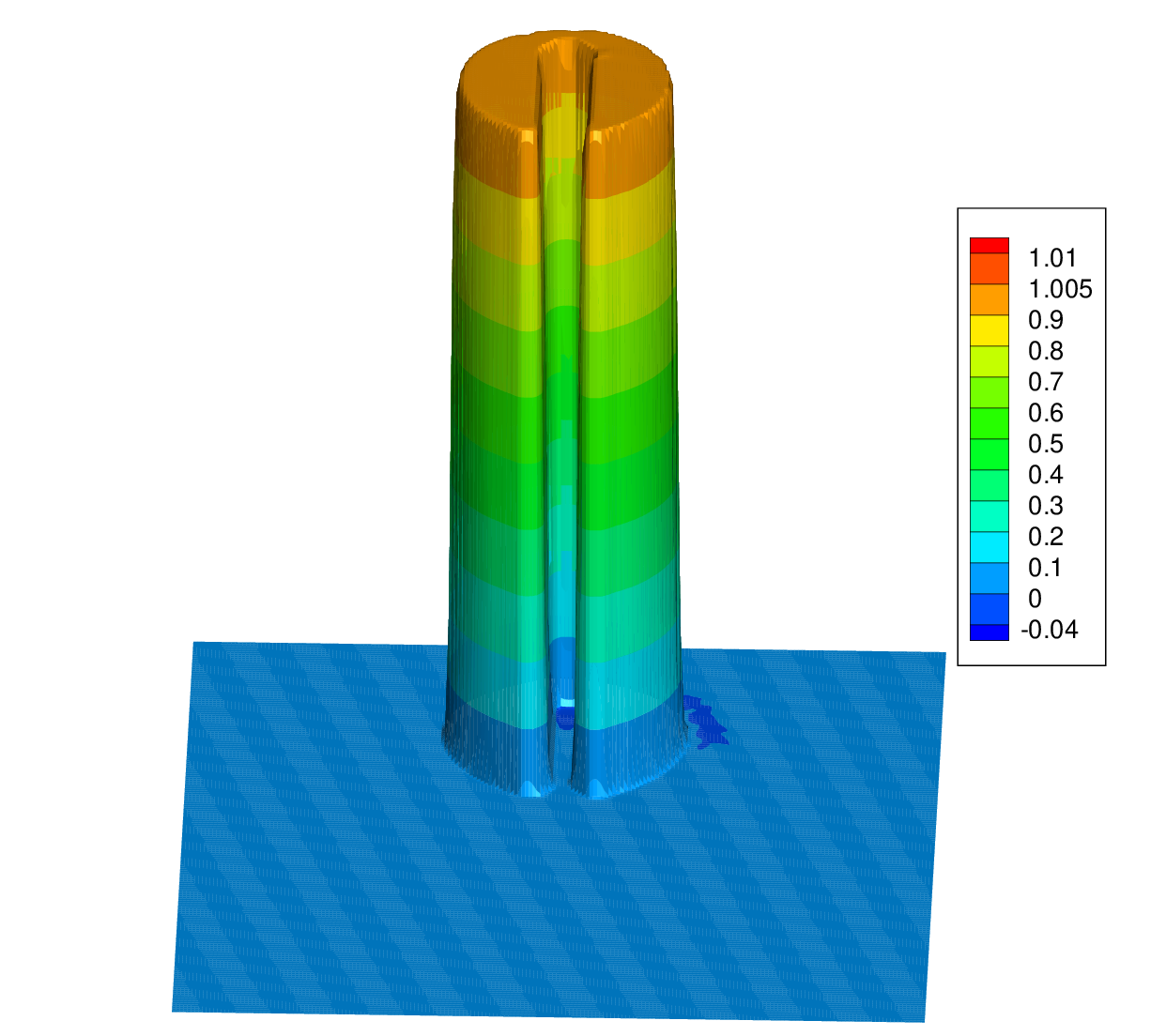}}
\caption{Comparisons between the original THINC scheme and the new THINC scheme which is less constrained by the CFL number. (a) 1D advection of a square wave. (b)(c) 2D Zalesak test.   \label{fig:advectionNewT}}
\par\end{centering}
\end{figure}

\section{Concluding remarks}\label{sec:conclusion}
This work proposes a novel method to evaluate and improve the convection boundedness of numerical schemes across discontinuities. We demonstrate that we can use the proposed method to determine the CFL conditions, identify the locations of over- and under-shoots, optimize the free parameters in the schemes, and develop new schemes with less stringent CFL conditions.

\section*{Acknowledgements}
XD and OKM acknowledge the funding provided by the Engineering and Physical Sciences Research Council and First Light Fusion through the ICL-FLF
Prosperity Partnership (grant number EP/X025373/1).  

\clearpage{}

\bibliographystyle{elsarticle-num-names.bst}
\bibliography{reference.bib}

\end{document}